\RequirePackage{fix-cm}
\documentclass[smallextended]{svjour3}       
\smartqed  
\usepackage{graphicx}
\usepackage{amssymb,amsmath,amsfonts,euscript}
\usepackage{fancyhdr}
\usepackage{epsfig}
\usepackage{setspace}
\usepackage{dsfont}
\usepackage[all]{xy}
\usepackage{amsmath}

\numberwithin{equation}{section}
\newcommand{\bal}{\begin{align}}
\newcommand{\eal}{\end{align}}
\newcommand{\beq}{\begin{equation}}
\newcommand{\eeq}{\end{equation}}

\newcommand{\bpf}{\begin{proof}}
\newcommand{\epf}{\end{proof}}
\newcommand{\bsp}{\begin{split}}
\newcommand{\esp}{\end{split}}
\newcommand{\bg}{\begin{gathered}}
\newcommand{\eg}{\end{gathered}}
\newtheorem{Thm}{Theorem}[section]
\newtheorem{Lem}{Lemma}[section]

\numberwithin{Thm}{section}
\numberwithin{Lem}{section}
\numberwithin{Cl}{section}
\numberwithin{Cor}{section}
\numberwithin{Def}{section}
\numberwithin{Rem}{section}
\numberwithin{Prop}{section}
\newcommand{\N}{\mathbb{N}}
\newcommand{\R}{\mathbb{R}}
\newcommand{\T}{\mathbb{P}}
\newcommand{\Z}{\mathbb{Z}}

\begin{document}

\title{A Central Limit Theorem for Non-Stationary Strongly Mixing Random Fields\thanks{$^2$ Supported partially by the NSA grant H98230-15-1-0006.}
}

\titlerunning{CLT for Non-Stationary Strongly Mixing Random Fields}        

\author{Richard C. Bradley$^1$  \and
        Cristina Tone$^2$
}


\institute{$^1$Department of Mathematics\\ Indiana University \\ Rawles Hall 313\\ Bloomington, Indiana 47405\\
              \email{bradleyr@indiana.edu}
           \and\\\\
           $^2$Department of Mathematics\\University of Louisville\\ 328 Natural Sciences Building\\ Louisville, Kentucky 40292\\
			  \email{cristina.tone@louisville.edu}
}

\date{Received: date / Accepted: date}

\maketitle

\begin{abstract}
In this paper we extend a central limit theorem of Peligrad for uniformly strong mixing random fields satisfying the Lindeberg condition in the absence of stationarity property. More precisely, we study the asymptotic normality of the partial sums of  uniformly $\alpha$-mixing non-stationary random fields satisfying the Lindeberg condition, in the presence of an extra dependence assumption involving maximal correlations.
\keywords{Central limit theorem\and non-stationary random fields\and strong mixing\and Lindeberg condition\and Kolmogorov's distance.}
\subclass{60F05 \and 60G60}
\end{abstract}

\section{Introduction}\label{s:intro}
In applications of statistics to data indexed by location, there is often an apparent lack of both stationarity and independence, but with a reasonable indication of ``weak dependence" between data whose locations are ``far apart". This has motivated a large amount of research on the theoretical question of to what extent central limit theorems hold for non-stationary random fields. This paper will examine that theoretical question for ``arrays of (non-stationary) random fields" under mixing assumptions analogous to those studied by Peligrad \cite{MP} in central limit theorems for ``arrays of random sequences".

Let $(\Omega, \mathcal{F}, \T)$ be a probability space.
For any two $\sigma$-fields $\mathcal{A}, \ \mathcal{B} \subseteq
\mathcal{F}$, define now the strong mixing coefficient
\begin{equation*}
\alpha(\mathcal{A}, \mathcal{B}):=\sup_{A \in
\mathcal{A}, B \in \mathcal{B}}|P(A\cap B)-P(A)P(B)| \end{equation*} and the
maximal coefficient of correlation
\begin{equation*}
\rho(\mathcal{A}, \mathcal{B}):=\sup|Corr(f, g)|, \text{ } f \in
L^2_{\text{real}}(\mathcal{A}), \text{ } g \in
L^2_{\text{real}}(\mathcal{B}).\end{equation*}
Suppose $d$ is a positive integer and $X:=(X_k, k\in
\Z^d)$ is not necessarily a strictly stationary random field.
In this context, for each positive integer $n$, define the following
quantity:
\begin{equation*}\label{s1.8}
\alpha(X, n):=\sup \alpha(\sigma(X_k, k\in Q), \sigma(X_k, k \in S)),\end{equation*}
where the supremum is taken over all pairs of
nonempty,
disjoint sets $ Q$, $S \subset \Z^d$ with the
following property: There exist $u \in \{1, 2, \ldots, d\}$ and $j \in \Z $ such that
$Q \subset \{k:=(k_1, k_2, \ldots, k_d) \in \Z^d: k_u \leq j \}$ and
$S \subset \{k:=(k_1, k_2, \ldots, k_d) \in \Z^d: k_u \geq j+n \}$.

The random field $X:=(X_k, k\in \Z^d)$ is said to be ``strongly
mixing" (or ``$\alpha$-mixing") if $\alpha(X, n)\rightarrow 0$ as
$n\rightarrow\infty$.

Also, for each positive integer $n$, define the following quantity:
\begin{equation*}\label{s1.9}\rho'(X, n):=\sup \rho(\sigma(X_k, k\in Q), \sigma(X_k, k \in S)),\end{equation*} where the supremum is taken over all pairs of nonempty,
finite disjoint sets $Q$, $S \subset \Z^d$ with the
following property: There exist $u \in \{1, 2, \ldots, d\}$ and nonempty disjoint sets $A$, $B \subset \Z$, with $dist(A, B):=\min_{a\in A, b\in B}|a-b|\geq n$ such that $Q \subset \{k:=(k_1, k_2, \ldots, k_d) \in \Z^d: k_u \in A \}$ and $S \subset \{k:=(k_1, k_2, \ldots, k_d) \in \Z^d: k_u \in B \}$.

The random field $X:=(X_k, k\in \Z^d)$ is said to be
``$\rho'$-mixing" if $\rho'(X, n)\rightarrow 0$ as
$n\rightarrow\infty$.

Again,  suppose $d$ is a positive integer. For a given random field $X:=(X_k, k \in \Z^d)$ and for  each $L:=(L_{1}, L_{2}, \ldots, L_{d}) \in \N^d$, define the ``box"
 \begin{align}\label{box}
 B(L):=\{ k:=(k_1, k_2, \ldots, k_d) \in \N^d: \forall u \in \{1, 2, \ldots, d\}, 1\leq k_u\leq L_u\}.
 \end{align}
 Obviously, the number of elements in the set $B(L)$ is $L_1 \cdot L_2 \cdot \ldots \cdot L_d$.

 For any given $L \in \N^d$ and any given ``collection" $X:=(X_k, k \in B(L))$, the dependence coefficients mentioned above can be defined for $n \in \N$ in the following way for convenience: one can trivially extend that collection $X$ to a random field $\widetilde{X}:=(X_k, k \in \Z^d)$ by defining $X_k=0$ for each $k \in \Z^d-B(L)$, and then one can define the dependence coefficients introduced in the previous section in the following way: for example, for $n \in \N$, $\rho'(X, n):=\rho'(\widetilde{X}, n)$.
 
We are interested in obtaining CLT's for non-stationary strongly mixing random fields, in the presence of an extra condition involving the maximal correlation coefficient $\rho'(X, n)$ defined above. 

Our main result presents a central limit theorem for sequences of random fields that satisfy a Lindeberg condition and uniformly satisfy both strong mixing and an upper bound less than $1$ on $\rho'(\cdot \ , 1)$, in the absence of stationarity.
There is no requirement of either a mixing rate assumption or the existence of moments of order higher than two.
The additional assumption of a uniform upper bound less than $1$ for $\rho'(\cdot \ , 1)$ cannot simply be deleted altogether from the theorem, even in the case of strict stationarity. For the case $d=1$, that can be seen from any
      (finite-variance) strictly stationary, strongly mixing
      counterexample to the CLT such that the rate of
      growth of the variances of the partial sums is at
      least linear; for several such examples,
      see e.g. \cite{Bradley3}, Theorem 10.25 and Chapters 30-33.
Our main theorem and an extension of it, given at the end of the paper, extend certain central limit theorems of Peligrad \cite{MP} involving ``arrays of random sequences".

The main result of this paper will be given in Theorem \ref{r}.
Then the material of this article will be divided as follows: Background results necessary in the proof of the main result will be given in Section 2. Sections 3, 4 and 5 will contain the proof of Theorem \ref{r}. More precisely, Section 3 will set up the induction assumption of the proof and contains two special cases introduced in Lemma \ref{l1}, respectively Lemma \ref{l2}, that imply our result. The general case will be presented in Lemma  \ref{l3}, which covers  Section 4 entirely. Section 5 of the paper will deal with the Lindeberg condition and the truncation argument.  Finally, Section 6 will state an extension of Theorem \ref{r} to a more general setup.

\begin{Thm}\label{r} Suppose $d$ is a positive integer. For each $n \in N$, suppose $L_n:=(L_{n1}, L_{n2}, \ldots, L_{nd})$ is an element of $N^d$, and suppose $X^{(n)}:=\left( X_k^{(n)}, k \in B(L_n) \right)$  is an array of random variables such that for each $k \in B(L_n)$,  $EX_{k}^{(n)}=0$ and  $E \left( X_{k}^{(n)} \right)^2<\infty$. Suppose the following mixing assumptions hold:
\begin{equation}\label{e1} \alpha(m):= \sup_n \alpha(X^{(n)}, m)\rightarrow 0 \text{ as } m\rightarrow\infty \text{ and}
\end{equation}
\begin{equation}\label{e1.1}\rho'(1):= \sup_n \rho'(X^{(n)},1)<1.  \eeq
For each $n \in \N$, define the random sum $S\left(X^{(n)}, L_n \right)=\sum_{k \in B(L_n)} X_k^{(n)}$, define the quantity $\sigma_n^2:=E(S(X^{(n)}, L_n))^2$, 
and assume that $\sigma_n^2>0$. Suppose also that the Lindeberg condition
\begin{equation}\label{e2} \forall \varepsilon>0,\  \lim_{n\rightarrow\infty} \frac{1}{\sigma_n^2}\sum_{k\in B(L_n)} E \left( X_{k}^{(n)} \right)^2I \left( \left | X_k^{(n)}\right|>\varepsilon\sigma_n\right)=0
 \end{equation} holds.
Then \[\sigma_n^{-1}S(X^{(n)}, L_n)\Rightarrow
N(0, 1) \text{ as } n\rightarrow\infty.\] 
\end{Thm}
(Here and throughout the paper $\Rightarrow$ denotes convergence in distribution.)

This result extends a theorem of Peligrad (see \cite{MP}, Theorem 2.2), which is Theorem \ref{r} for the case $d=1$. Later on, Peligrad and Utev \cite{UP} obtained an invariance principle for random elements associated to sums of strongly mixing triangular arrays of random variables associated with the interlaced mixing coefficients $\rho^*_n$. Their invariance principle generalizes the corresponding results for independent random variables treated e.g. by Prohorov \cite{Pro}. For the strictly stationary case see Peligrad \cite{Pelig2}. 
 
For a sequence of strictly stationary random fields that are uniformly $\rho'$-mixing and satisfy a Lindeberg condition, a central limit theorem is obtained  in \cite{Tone4} for sequences of ``rectangular" sums from the given random fields. The ``Lindeberg CLT" is then used to prove a CLT for some kernel estimators of probability density for some strictly stationary random fields satisfying $\rho'$-mixing, and whose probability density and joint densities are absolutely continuous, generalizing the results in \cite{Mil}, under $\rho^*$-mixing.

\section{Background Results}

The proof of Theorem \ref{r} uses frequently the following results. The first one is a consequence of Theorem 28.10(I) \cite{Bradley3}
which gives an upper bound for the variance of partial sums.

\begin{Thm}\label{br0}
Suppose $d$ is a positive integer, $L \in \N^d$, and $X:=\left( X_k, k \in B(L) \right)$  is a (not necessarily strictly stationary) random field such that for each $k\in B(L)$,  the random variable $ X_k$ has mean zero and finite second moments. Suppose $\rho'(X, j)<1$ for some $j \in \N$. Then 
for any nonempty finite set $S \subseteq B(L)$,
\beq\label{br0.1} E  \left | \sum_{k \in S} X_k  \right |^2
\leq C \sum_{k \in S} E \left (X_k  \right )^2, \eeq where $C:=j^d \left (1+ \rho'(X, j) \right)^d/\left (1- \rho'(X, j) \right )^d.$
\end{Thm}

The second result is a consequence of Theorem 28.9 \cite{Bradley3}
which gives lower and upper bounds for the variance of partial sums.

\begin{Thm}\label{br1}
Suppose $d$ is a positive integer, $L \in \N^d$, and $X:=\left( X_k, k \in B(L) \right)$  is a (not necessarily strictly stationary) random field such that for each $k\in B(L)$,  the random variable $ X_k$ has mean zero and finite second moments. Suppose $\rho'(X,1)<1$. Then for any nonempty finite set $S \subseteq B(L)$,
\begin{equation}\label{bvs}
C^{-1} \sum_{k \in S} E  \left | X_k  \right |^2 \leq E \left| \sum_{k \in S} X_k \right|^2 \leq C \sum_{k \in S}E  \left |X_k \right |^2,
\end{equation}
$ \text{where }C:=(1+ \rho'(X, 1))^d / (1- \rho'(X, 1))^d$.
\end{Thm}

The  next  result used is a particular case of the Rosenthal inequality (see Theorem 29.30, \cite{Bradley3}) for the exponent 4.

\begin{Thm}\label{br2}
Suppose $d$ and $m$ are each a positive integer and $r \in [0, 1)$. Then there exists a constant $C:=C(d, 4, r, m)$ such that the following holds: 

Suppose $L \in \N^d$ and  $X:=\left( X_k, k \in B(L) \right)$ is a (not necessarily strictly stationary) random field such that for each $k \in B(L)$, $EX_k=0$ and $E|X_k|^4<\infty$,  and $\rho'(X, m)\leq r$. Then for any nonempty finite set $S \subseteq B(L)$, one has that
\begin{equation}\label{ros}
 E \left| \sum_{k \in S} X_k\right|^4 \leq C \cdot \left[ \sum_{k \in S}E  \left |X_k \right |^4+\left( \sum_{k \in S}E  \left |X_k\right |^2 \right)^2 \right].
\end{equation}
\end{Thm}

\section{Induction Assumption}
The proof of Theorem \ref{r} will be done by induction on $d$. For $d=1$, Theorem \ref{r} was proved by Peligrad (\cite{MP}, Theorem 2.2). Now suppose $d$ is an integer such that $d \geq 2$. As the induction hypothesis, suppose Theorem \ref{r} holds in the case where $d$ is replaced by the particular integer $d-1$. To complete the induction step (and thereby the proof of Theorem \ref{r}), it suffices to prove Theorem \ref{r} in the case of the given integer $d$.

 To carry out the induction step, we will first treat the case where 
\beq\label{inf} \inf_{n\in \N} \sigma_n^2>0
\eeq and
\beq\label{strong}
\theta_n:=\sup_{k \in B(L_n)} \left\| X_k^{(n)} \right\|_{\infty} \rightarrow 0.
\eeq
Notice that \eqref{strong} (together with \eqref{inf}) implies the Lindeberg condition \eqref{e2}.
Our goal in Sections 3 and 4 is to show that for $X^{(n)}:=\left( X_k^{(n)}, k \in B(L_n) \right)$ satisfying \eqref{e1}, \eqref{e1.1}, \eqref{inf}, and \eqref{strong}, the CLT holds, that is 
\beq\label{goal}
 \frac{1}{\sigma_n}\sum_{k \in B(L_n)} X_k^{(n)} \Rightarrow
N(0, 1) \text{ as } n\rightarrow\infty.
\eeq
Then in Section 5, the induction argument will be completed with the use of a standard truncation argument to reduce to the case of the restrictions \eqref{inf}-\eqref{strong}.

In what follows, for convenience, we shall use the notation $L_n:=L^{(n)}:=\left( L_1^{(n)},  L_2^{(n)}, \ldots,  L_d^{(n)} \right)$.
\begin{Lem}\label{l1}
Suppose in addition to the properties  \eqref{e1}, \eqref{e1.1}, \eqref{inf}, and \eqref{strong} that $\sup_{n \in \N} L_1^{(n)}<\infty$.  
For each $n \geq 1$, define the element $ \widetilde{L}^{(n)} \in \N^{d-1}$ by $\widetilde{L}^{(n)}:=\left( L_2^{(n)},  L_3^{(n)}, \ldots,  L_d^{(n)} \right)$. For each $n \geq 1$, define the random field $W^{(n)}:=\left( W^{(n)}_k, k \in B(\widetilde{L}^{(n)}) \right)$ as follows:
For each $k:=(k_2, k_3, \ldots, k_d) \in B(\widetilde{L}^{(n)})$,
\[ W_{k}^{(n)}:=\sum_{u \in \left\{ 1, 2,  \ldots, L_1^{(n)}  \right \} }X^{(n)}_{ \left( u, \ k\right) }.
\] 
Then 
\[ \frac{1}{\sigma_n}\sum_{k \in B(L^{(n)})} X_k^{(n)}= \left( E \left( \sum_{k \in B(\widetilde{L}^{(n)})} W_k^{(n)} \right)^{2} \right)^{-1/2} \sum_{k \in B(\widetilde{L}^{(n)})} W_k^{(n)}  \Rightarrow
N(0, 1) \]
$ \text{as } n\rightarrow\infty$.
\end{Lem}

\begin{proof} It is easy to see that \[E \left( \sum_{k \in B(\widetilde{L}^{(n)}) }W_k^{(n)} \right)^{2} =E \left( \sum_{k \in B(\widetilde{L}^{(n)}) }\sum_{u=1}^{L_1^{(n)}} 
X_{(u,k)}^{(n)} \right)^{2}=\sigma_n^2.
\]
The random field $ W^{(n)}$ inherits the properties from the parent random field $ X^{(n)}$, that is, the mixing and the moment properties. In addition, 
\begin{align*} \bsp & \sup_{k \in B(\widetilde{L}^{(n)}) } \| W_k^{(n)}\|_{\infty}= \sup_{k \in B(\widetilde{L}^{(n)}) }  \| \sum_{u=1}^{L_1^{(n)}}  X_{(u,k)}^{(n)}  \|_{\infty} \leq \sup_{k \in  B(\widetilde{L}^{(n)}) }  \sum_{u=1}^{L_1^{(n)}}   \| X_{(u,k)}^{(n)}  \|_{\infty} \\
& \leq    \sum_{u=1}^{L_1^{(n)}} \sup_{k \in  B(\widetilde{L}^{(n)}) }  \| X_{(u,k)}^{(n)}  \|_{\infty} 
\leq  \sum_{u=1}^{L_1^{(n)}} \sup_{k \in B(L^{(n)} )}  \| X_k^{(n)}  \|_{\infty} 
=  L_1^{(n)} \theta_n \rightarrow 0 \text{ as } n\rightarrow \infty. 
 \esp
\end{align*}
By the induction hypothesis for $d-1$, the CLT holds, and the proof of Lemma \ref{l1} is complete.
\end{proof}

\begin{Lem}\label{l2}
Suppose that $ L_1^{(n)} \rightarrow \infty$ as $n\rightarrow \infty$ together with the properties mentioned earlier, namely,  \eqref{e1}, \eqref{e1.1}, \eqref{inf}, and \eqref{strong}.  For $\forall n \in \N, \ \forall j \in \{1, 2, \ldots, L_1^{(n)}\}$,  let us define the random variable \[ Y_j^{(n)}=\sum_{\{ k=(k_1, \ldots, k_d) \in B(L^{(n)}) : k_1=j \} } X^{(n)}_k.
\] 
Assume also that \beq\label{l2.0} \sup_{j \in \{1, 2, \ldots, L_1^{(n)} \} } \left( s_j^{(n)} \right)^2 \rightarrow 0 \text{ as } n\rightarrow \infty, \text{ where }   \left( s_j^{(n)} \right)^2= E \left( Y_j^{(n)} \right)^2.\eeq
Then
\beq\label{l2g} \frac{1}{\sigma_n}\sum_{k \in B(L^{(n)})} X_k^{(n)}=\frac{1}{\sigma_n} \sum_{j=1}^{L_1^{(n)}} Y_j^{(n)}  \Rightarrow
N(0, 1) \text{ as } n\rightarrow\infty. \eeq
\end{Lem}

\begin{proof}
We shall first give some notations and basic observations that will be used in both the main argument below for Lemma \ref{l2} and the argument for Lemma \ref{l3} in Section 4. 

For each $n \in \N$ and each $j \in \left\{1, 2, \ldots, L_1^{(n)} \right\}$, define the (``slice") set 
\[ \text{slice}_j^{(n)}:=\left\{k:=(k_1,  \ldots, k_d) \in B(L^{(n)}) : k_1=j \right\}. \]
Then for each such $n$ and $j$, $Y_j^{(n)}=\sum_{k \in \text{slice}_j^{(n)} } X^{(n)}_k$. By Theorem \ref{br1}, for each such $n$ and $j$, the two numbers $\left( s_j^{(n)} \right)^2=E \left( Y_j^{(n)} \right)^2= E \left( \sum_{k \in \text{slice}_j^{(n)} } X^{(n)}_k   \right)^2$ and $\sum_{k \in \text{slice}_j^{(n)} }E \left(  X^{(n)}_k   \right)^2$ either are both $0$ or are both positive and within a constant factor (in $[c^{-1}, c]$, where $c :=(1+\rho'(1) )^d/ (1-\rho'(1))^d $) of each other. Similarly, by \eqref{inf} and Theorem \ref{br1}, for each $n \in \N$, the following three quantities  are positive and are within a constant factor (in the same interval $[c^{-1}, c]$) of each other:
\begin{align*} \label{equalities} \bsp &  \sigma_n^2= E \left(  \sum_{k \in B(L_n)} X_k^{(n)}\right)^2 = E \left(  \sum_{j=1}^{L_1^{(n)}} Y_j^{(n)}\right)^2;\\
& \sum_{j=1}^{L_1^{(n)} } \left( s_j^{(n)}\right)^2= \sum_{j=1}^{L_1^{(n)} }  E \left(  Y_j^{(n)} \right)^2=\sum_{j=1}^{L_1^{(n)} }  E \left(  \sum_{k \in \text{slice}_j^{(n)} } X_k^{(n)} \right)^2; \\
&  \sum_{j=1}^{L_1^{(n)} }   \sum_{k \in \text{slice}_j^{(n)} } E \left(  X_k^{(n)} \right)^2 =\sum_{k \in B(L^{(n)})}E \left(  X_k^{(n)} \right)^2.\\
\esp
\end{align*}
Finally, by \eqref{inf}, $\sigma^2_n \ll\sigma_n^4$ as $ n \rightarrow \infty$. Here and below, the notation ``$\ll$" means $O(\ldots)$.

To prove \eqref{l2g}, the main task will be to show that Lyapounov's condition holds (with exponent $4$), that is, 
\beq\label{lyap}
\lim_{n \rightarrow \infty} \frac{1}{\sigma_n^4} \sum_{j=1}^{L_1^{(n)} }E \left(  Y_j^{(n)} \right)^4=0.
\eeq
For each $n \in \N$, applying \eqref{e1.1} and Theorem \ref{br2} (and using its constant $C$) and then adding up over all  $j \in \left\{1, 2, \ldots, L_1^{(n)} \right\}$, we obtain that
\begin{align}\label{l2.1} \bsp
&  \sum_{j=1}^{L_1^{(n)} } E \left(  Y_j^{(n)} \right)^4 
\leq C \left[   \sum_{j=1}^{L_1^{(n)} } \sum_{k \in \text{slice}_j^{(n)}}  E \left( X_k^{(n)} \right)^4   +   \sum_{j=1}^{L_1^{(n)} }  \left(   \sum_{k \in \text{slice}_j^{(n)} } E \left( X_k^{(n)} \right)^2 \right)^2  \right]
\esp
\end{align}
Using \eqref{strong} and  Theorem \ref{br1},  the first term in the right-hand side of \eqref{l2.1} can be bounded above in the following way:
\begin{align*} \bsp&   \sum_{j=1}^{L_1^{(n)} } \sum_{k \in \text{slice}_j^{(n)}}  E \left( X_k^{(n)} \right)^4  =
 \sum_{j=1}^{L_1^{(n)} } \sum_{k \in \text{slice}_j^{(n)}}  E \left[ \left( X_k^{(n)}  \right)^2 \left( X_k^{(n)}  \right)^2 \right]\\
 & \leq \theta_n^2  \sum_{j=1}^{L_1^{(n)} } \sum_{k \in \text{slice}_j^{(n)}}  E \left( X_k^{(n)} \right)^2 \ll \theta_n^2 \sigma_n^2 \ll \theta_n^2 \sigma_n^4=o(\sigma_n^4)   \text{ as } n \rightarrow \infty.
 \esp.
\end{align*}
By \eqref{l2.0} (and the fact $\sigma_n^2 \ll \sigma_n^4$), the second term in the right-hand side of \eqref{l2.1} can be bounded above in the following way:
\begin{align*}\bsp &
 \sum_{j=1}^{L_1^{(n)} }  \left(   \sum_{k \in \text{slice}_j^{(n)} } E \left( X_k^{(n)} \right)^2 \right)^2  = \sum_{j=1}^{L_1^{(n)} }  \left(   \sum_{k \in \text{slice}_j^{(n)} } E \left( X_k^{(n)} \right)^2 \right)  \left(   \sum_{k \in \text{slice}_j^{(n)} } E \left( X_k^{(n)} \right)^2 \right) \\
&\ll   \left[ \sup_{j \in \{1, 2, \ldots, L_1^{(n)} \} }  \left( s_j^{(n)} \right)^2 \right]  \sum_{j=1}^{L_1^{(n)} }   \sum_{k \in \text{slice}_j^{(n)} } E \left( X_k^{(n)} \right)^2 \\
&\ll \left[  \sup_{j \in \{1, 2, \ldots, L_1^{(n)} \} } \left( s_j^{(n)} \right)^2 \right] \sigma_n^2 =o(\sigma_n^4)  \text{ as } n \rightarrow \infty.
\esp
\end{align*}
Hence, \eqref{lyap} holds, and as a consequence, the Lindeberg condition is satisfied.  Applying Peligrad's CLT for $d=1$ (see \cite{MP}, Theorem 2.2) to the array  $\left( Y_j^{(n)},\ n\in \N, \ j \in \left\{1, 2, \ldots, L_1^{(n)} \right\} \right)$, one has that \eqref{l2g} holds. The proof of Lemma \ref{l2} is complete.
\end{proof}

\section{``General Lemma"}
The following lemma deals with the most general case under the restrictions \eqref{inf} and \eqref{strong}.

\begin{Lem}\label{l3}
Suppose that for each $n \in \N$, $L_n \in \N^d$, $X^{(n)}:=\left( X_k^{(n)}, k \in B(L_n) \right)$  is a (not necessarily strictly stationary) random field such that for each $k\in B(L_n)$, $ X_k^{(n)}$ has mean zero and finite second moment. Suppose that \eqref{e1}, \eqref{e1.1},  \eqref{inf}, and \eqref{strong} are satisfied. Then 
\begin{align*}
 \frac{1}{\sigma_n}\sum_{k \in B(L_n)} X_k^{(n)} \Rightarrow
N(0, 1) \text{ as } n\rightarrow\infty.
\end{align*}
\end{Lem}

\bpf
It suffices to show that for an arbitrary fixed infinite set $S \subseteq \N$, there exists an infinite set $T \subseteq S$ such that 
\begin{align}\label{l3g}
 \frac{1}{\sigma_n}\sum_{k \in B(L_n)} X_k^{(n)} \Rightarrow
N(0, 1) \text{ as } n\rightarrow\infty, \ n \in T.
\end{align}
\noindent Again we write $L_n$ as $L^{(n)}:=\left(L_1^{(n)}, L_2^{(n)}, \ldots, L_d^{(n)} \right)$. We freely use the notations $Y_j^{(n)}, $$\left( s_j^{(n)} \right)^2$ and $\text{slice}_j^{(n)}$ from Lemma \ref{l2} and its proof. The observations in the first part of the proof of Lemma \ref{l2} (that is, prior to the paragraph containing equation \eqref{lyap}) hold in our context here, and will be used freely. (Of course the convergence to $0$ in \eqref{l2.0} is not assumed, and  may not hold, in our context here.) Applying those observations, without loss of generality (that is, without sacrificing \eqref{inf} or \eqref{strong}) we now normalize so that
 \beq\label{normalize} \forall n \geq 1, \ \sum_{j=1}^{L_1^{(n)} }  \left( s_j^{(n)} \right)^2=1.\eeq
The proof of  \eqref{l3g} (including the choice of an appropriate infinite set $T \subseteq S$) will be divided into twelve ``steps".\\

{\bf{Step 1: }}  Consider first the case where  $\sup_{n \in S} L_1^{(n)}< \infty$. By Lemma \ref{l1}, the asymptotic normality in \eqref{l3g} holds with $T:=S$, and for this case we are done.\\

{\bf{Step 2: }} Now henceforth suppose that $\sup_{n \in S} L_1^{(n)}= \infty$. 

Let us choose an infinite set $S_0 \subseteq S$ be such that $L_1^{(n)}\rightarrow\infty
\text{ as } n\rightarrow\infty, \ n \in S_0$. For each $n \geq 1$,  let $p(n, j)$, $j \in \{1, 2, \ldots, L_1^{(n)} \} $ be a  permutation of the set $ \{1, 2, \ldots, L_1^{(n)} \} $ such that 
\beq\label{perm}
\left( s_{p(n, 1)}^{(n)} \right)^2 \geq \left( s_{p(n, 2)}^{(n)} \right)^2 \geq \ldots \geq \left( s_{p(n, L_1^{(n)} )}^{(n)} \right)^2.
\eeq
By \eqref{normalize}, we obtain that 
\beq\label{norm}
\sum_{j=1}^{L_1^{(n)} }  \left( s_{p(n, j)}^{(n)} \right)^2=1.
\eeq
As a consequence, by \eqref{perm} and \eqref{norm}, 
\beq\label{l3.1}
\forall n \geq 1,\  \forall j \in \{1, 2, \ldots, L_1^{(n)} \}, \  \left( s_{p(n, j)}^{(n)} \right)^2 \leq \frac{1}{j}.
\eeq
Of course since $L_1^{(n)} \rightarrow \infty$ as $n \rightarrow \infty, \ n \in S_0$, one has that for each $l\geq1$, the index $p(n, l)$ and the number $ \left( s_{p(n, l)}^{(n)} \right)^2$ are defined for all sufficiently large $n \in S_0$. That will be used repeatedly in what follows.

Let us now define the following infinite sets: 
\[ S_1 \subseteq S_0 \text{ such that } \lambda_1=\lim_{ n \rightarrow \infty, \ n \in S_1} \left( s_{p(n, 1)}^{(n)} \right)^2 \text{ exists}; \]

\[ S_2 \subseteq S_1 \text{ such that } \lambda_2=\lim_{ n \rightarrow \infty, \ n \in S_2} \left( s_{p(n, 2)}^{(n)} \right)^2 \text{ exists}; \]

\[ S_3 \subseteq S_2 \text{ such that } \lambda_3=\lim_{ n \rightarrow \infty, \ n \in S_3} \left( s_{p(n, 3)}^{(n)} \right)^2 \text{ exists}; \]
and so on. By the Cantor diagonalization method,  we obtain an infinite set $S_{00}:=\{ \widetilde{n}_1<\widetilde{n}_2<\widetilde{n}_3< \ldots \} $ such that $\widetilde{n}_l \in S_l$ and $S_l \supseteq \{\widetilde{n}_l, \widetilde{n}_{l+1}, \widetilde{n}_{l+2}, \ldots \}$. 
For the resulting infinite set $S_{00}$, one has that  
$S_{00} \subseteq S_0 \subseteq S$, and by \eqref{perm} one also has that 
\beq\label{l3.2}
\forall l\geq 1,\  \lim_{ n \rightarrow \infty, \ n \in S_{00}} \left( s_{p(n, l)}^{(n)} \right)^2 =\lambda_l; \text{ with  } \lambda_1 \geq \lambda_2 \geq \lambda_3 \ldots.
\eeq
In addition,  $\forall m\geq 1$, one has by \eqref{norm} that  
$ \sum_{j=1}^{m }  \left( s_{p(n, j)}^{(n)} \right)^2\leq 1 \text{ for all } n \in S_{00}$ sufficiently large such that $L_1^{(n)} \geq m$; and hence for every $m \geq 1$, 
$ \sum_{j=1}^m \lambda_j \leq 1 $ by \eqref{l3.2}. Hence
\beq\label{lambda}
\lambda:=\sum_{j=1}^{\infty} \lambda_j \leq 1.\\
\eeq

{\bf{Step 3: }} Consider first the case where  $\lambda=0$. Then $\lambda_j=0$ for all $ j \geq1$.  By \eqref{l3.1}, \eqref{l3.2}, and a simple argument,   $\sup_{j \in \{1, 2, \ldots, L_1^{(n)} \} } \left( s_{p(n, j)}^{(n)} \right)^2 \rightarrow 0 \text{ as } n\rightarrow \infty, \ n \in S_{00}.$ By Lemma \ref{l2}, 
\beq\label{l3gs}\frac{1}{\sigma_n} \sum_{j=1}^{L_1^{(n)}} Y_j^{(n)}=\frac{1}{\sigma_n}\sum_{k \in B(L^{(n)})} X_k^{(n)} \Rightarrow
N(0, 1) \text{ as } n\rightarrow\infty, \ n \in S_{00}. \eeq
Thus \eqref{l3g} holds with $T:=S_{00}$, and for this case we are done.\\

{\bf{Step 4: }} Now henceforth suppose that 
$ \lambda >0$. (Then by \eqref{l3.2}  and \eqref{lambda},  $\lambda_1>0$.) Our task now is to show that  \eqref{l3g} holds for some infinite set $T \subseteq S_{00}$. 

Recall again that $L_1^{(n)}  \rightarrow\infty \text{ as }n\rightarrow\infty, \ n \in S_{00}$. For each $q \geq 1$ and  each $ n \in S_{00}$ such that $L_1^{(n)} >q$, define the set
\begin{align*}
\overline{\Gamma}_1^{(q, n)}=\{ p(n, 1), p(n, 2), \ldots, p(n, q)\} 
\end{align*}
and the random variable
\begin{align*}
W^{(q, n)}:= \sum_{j \in \overline{\Gamma}_1^{(q, n)} } Y_j^{(n)}.
\end{align*}

Recall that (here in Step 4 and henceforth) $\lambda_1>0$. By \eqref{l3.2}, $E \left( Y_{p(n, 1)}^{(n)} \right)^2=\left( s_{p(n, 1)}^{(n)} \right)^2>\lambda_1/2$ for all $n \in S_{00}$ sufficiently large.

For each positive integer $q$, the following observations hold: Trivially,  we have that $\sum_{j \in \overline{\Gamma}_1^{(q, n)}} E \left( Y_j^{(n)} \right)^2  \geq E \left( Y_{p(n, 1)}^{(n)} \right)^2 \geq \lambda_1/2 $ for all $n \in S_{00}$ sufficiently large. Hence, by Theorem \ref{br1}, there exists a positive number $c_0$ (not even depending on $q$) such that $E \left( W^{(q, n)} \right)^2\geq c_0$ for all $n \in S_{00}$ sufficiently large.
That is the analog of \eqref{inf} for   sufficiently large $n \in S_{00}$ when the indices $k:=(k_1, \ldots, k_d) \in B(L^{(n)})$ are restricted to the ones such that  $k_1 \in \overline{\Gamma}_1^{(q, n)}$. Hence, one can apply Lemma \ref{l1}, and one obtains that 
\begin{align*}
\frac{W^{(q, n)}} { \|W^{(q, n)} \|_2} \Rightarrow
N(0, 1) \text{ as } n\rightarrow\infty, \ n \in S_{00}. 
\end{align*}
The convergence above was shown for arbitrary $q\geq 1$. By a well known theorem for continuous limiting distributions, one now has that
\begin{align*}
\forall q \geq 1,\  \sup_{x\in \R} \left|  F_{ 
W^{(q, n)} / \|W^{(q, n)} \|_2}(x
) -\Phi(x)  \right| \rightarrow
0 \text{ as } n\rightarrow\infty, \ n \in S_{00}. 
\end{align*}
Here $\Phi(x)$ represents the distribution function of a $N(0, 1)$ random variable and $F_V$
 is the  distribution function of a given random variable $V$.\\
 
{\bf{Step 5: }}  For each $q \geq 1$, let $m_q \in \N$ be such that
 \beq\label{c1}\alpha(m_q)<\frac{1}{q^2}.\eeq 
 Let $n_1<n_2< \ldots \in S_{00}$ be such that  for all $q \geq 1$, the following hold:
 \beq\label{c2}
 L_1^{(n_q)}>q^2m_q;
 \eeq
 
  \beq\label{c3}
\|W^{(q, n_q)} \|_2>0 \text{ and }  \sup_{x\in \R} \left|  F_{ 
W^{(q, n_q)} / \|W^{(q, n_q)} \|_2}(x
) -\Phi(x)  \right| \leq \frac{1}{q}, \text{ and}
 \eeq
 
 \beq\label{c4}
 \left|   \sum_{j=q+1}^{q^2m_q}   \left(  s_{p(n_q, j)}^{(n_q)} \right)^2   -    \sum_{j=q+1}^{q^2m_q} \lambda_j  \right| \leq \frac{1}{q}.
  \eeq
(To justify \eqref{c4}, see \eqref{l3.2}.)\\

\noindent For each $q \geq 1$, define the following four index sets:
 
 \begin{equation}\label{gamma}
\begin{cases}
\Gamma_1^{(q)}=\{ p(n_q, 1), p(n_q, 2), \ldots, p(n_q, q)\}, \\ 

\Gamma_2^{(q)}=\{ p(n_q, q+1), p(n_q, q+2), \ldots, p(n_q, q^2m_q)\}, \\ 

\Gamma_3^{(q)}=\{  j \in \{1, \ldots, L_1^{(n_q)} \} -\{ \Gamma_1^{(q)} \cup \Gamma_2^{(q)}  \} | \  \exists i \in \Gamma_1^{(q)} \text{ such that } \   |i-j|\leq m_q \},\\ 

\Gamma_4^{(q)}=\{  j \in \{1, \ldots, L_1^{(n_q)} \} -\{ \Gamma_1^{(q)} \cup \Gamma_2^{(q)}  \} | \ \forall i \in \Gamma_1^{(q)} ,  \ |i-j|> m_q \}.  
\end{cases}
\end{equation}
For each $q \geq 1$, those four sets in \eqref{gamma} form a partition of the set $ \left \{1, 2, \ldots, L_1^{(n_q)} \right\}$ (see \eqref{c2}). For a given $q \geq 1$, one of the latter two sets $\Gamma_3^{(q)}$, 
$\Gamma_4^{(q)}$ could perhaps be empty. Note that for each  $q \geq 1$, the set $\Gamma_1^{(q)}$ here is the set $\overline{\Gamma}_1^{(q, n_q)}$ in the notations in Step 4.

 For each $q \geq 1$ and each $i \in \{1, 2, 3, 4\}$, define the random variable 
\begin{equation}\label{l3.4}
U_i^{(q)}=\sum_{j \in \Gamma_i^{(q)}} Y_j^{(n_q)}.
\end{equation}
Note that  for each $q \geq 1$, $U_1^{(q)}=W^{(q, n_q)}$ by \eqref{l3.4} (see Step 4), and also
\begin{equation}\label{l3.5}
\sum_{i=1}^4 U_i^{(q)}=\sum_{j=1}^{L_1^{(n_q)}} Y_j^{(n_q)}=\sum_{k \in B \left( L^{(n_q)} \right)} X_k^{(n_q)} .
\end{equation}

{\bf{Step 6: }}  Notice that due to \eqref{e1.1}, Theorem \ref{br1}, 
followed by \eqref{normalize}, we obtain that  for each $q \geq 1$,
\begin{align*}\label{l3.6}  \bsp &0\leq \left( \frac{1-\rho'(1)}{1+\rho'(1)}  \right) \sum_{j \in \Gamma_1^{(q)} }  E \left( Y_j^{(n_q)} \right)^2 \leq E \left( U_1^{(q)} \right)^2 \leq   \left( \frac{1+\rho'(1)}{1-\rho'(1)}  \right) \sum_{j \in \Gamma_1^{(q)} }  E \left( Y_j^{(n_q)} \right)^2\\
& \hspace{7.5cm} \leq  \left( \frac{1+\rho'(1)}{1-\rho'(1)}  \right)  < \infty.
\esp
\end{align*}
Similarly,  for each $q \geq 1$,
\begin{align*} \bsp 
&0 \leq E \left( U_4^{(q)} \right)^2 \leq   \left( \frac{1+\rho'(1)}{1-\rho'(1)}  \right) < \infty.
\esp
\end{align*}
Hence, there exists an infinite set $T  \subseteq \N$ such that 
\[
\eta_1^2:=\lim_{q \rightarrow \infty, \ q \in T} E \left( U_1^{(q)} \right)^2 \text{ exists (in } \R) \text{, and}
\]
\[
\eta_4^2:=\lim_{q \rightarrow \infty, \ q \in T} E \left( U_4^{(q)} \right)^2 \text{ exists (in } \R).
\]

Our goal now is to prove that for the infinite set $T$ just specified here, 
\[ \sigma_{n_q}^{-1} \sum_{k \in B\left( L^{(n_q)} \right) } X_k^{(n_q)} \Rightarrow N(0, 1) \text{ as } q \rightarrow \infty, \ q \in T.\]
That will accomplish \eqref{l3g} (and therefore complete the proof of Lemma \ref{l3}) with the set $T$ in \eqref{l3g} replaced here by the set $\{n_q:  q \in T\}$, which is an infinite subset of $S_{00}$ and hence of $S$.

In what follows, the ``$N(0, 0)$ distribution" will of course mean the degenerate ``point mass at 0". It will be tacitly kept in mind and used freely that if a sequence of random variables converges to 0 in the  2-norm, then it converges to 0 in probability and hence converges to $N(0, 0)$ in distribution.

{\bf{Step 7: }} ``The asymptotic normality of $U_1^{(q)}$". By \eqref{c3}, we obtain that
\[ \sup_{x\in \R} \left| F_{ 
W^{(q, n_q)} / \|W^{(q, n_q)} \|_2}(x
)-\Phi(x)  \right|  \rightarrow 0 \text{ as } q \rightarrow \infty, \text{ hence} \]
\[  \frac{ U_1^{(q)} } {  \left \|  U_1^{(q)} \right\|_2 } \Rightarrow N(0, 1)  \text{ as } q \rightarrow \infty, \ q \in T. \]
So, we obtain the asymptotic normality of the random variable $U_1^{(q)}$, namely
\beq\label{au1}
 U_1^{(q)} \Rightarrow N(0, \eta_1^2)  \text{ as } q \rightarrow \infty, \ q \in T.\\
\eeq

{\bf{Step 8: }} ``The asymptotic normality of $U_4^{(q)}$".
Recall from \eqref{c2} that $L_1^{(n_q)} \rightarrow \infty$ as $ q \rightarrow \infty$.
In addition, by \eqref{l3.1} and the definition of $\Gamma_4^{(q)}$ in \eqref{gamma}, 
\[ \sup_{j \in \Gamma_4^{(q)} } E \left( Y_j^{(n_q)} \right)^2  \leq \frac{1}{q^2 m_q+1} \rightarrow 0\text{ as } q \rightarrow \infty, \ q \in T.       \]
Trivially if $\eta_4^2=0$, or if instead $\eta_4^2>0$ then by Lemma \ref{l2} (with the indices $k:=(k_1, \ldots, k_d) \in B\left( L^{(n_q)} \right)$ restricted to the ones such that $k_1 \in \Gamma_4^{(q)}$), one has that 
\beq\label{au4}
 U_4^{(q)} \Rightarrow N(0, \eta_4^2)  \text{ as } q \rightarrow \infty, \ q \in T.
\eeq

{\bf{Step 9: }}  ``Negligibility of $U_2^{(q)}$".
By \eqref{lambda}, 
\[
\sum_{j=q+1}^{\infty} \lambda_j  \rightarrow 0 \text{ as } q \rightarrow \infty, \ q \in T. 
\]
Therefore, 
\[
\sum_{j=q+1}^{q^2m_q} \lambda_j  \rightarrow 0 \text{ as } q \rightarrow \infty, \ q \in T, 
\]
which gives us by \eqref{c4} that
\[
\sum_{j=q+1}^{q^2m_q}E \left( Y_{p(n_q, j)}^{(n_q)} \right)^2  \rightarrow 0 \text{ as } q \rightarrow \infty, \ q \in T. 
\]
As a consequence, referring to \eqref{gamma} and \eqref{l3.4} and bounding above the second moment of  the random variable $U_2^{(q)}$ using  Theorem \ref{br1} , we obtain that
\[
E \left(   \sum_{j \in \Gamma_2^{(q)} } Y_j^{(n_q)}  \right)^2   \rightarrow 0\text{ as } q \rightarrow \infty, \ q \in T, 
\]
hence
\beq\label{negu2}
 U_2^{(q)}    \rightarrow 0 \text{ in probability  as } q \rightarrow \infty, \ q \in T.
\eeq

{\bf{Step 10: }} ``Negligibility of $U_3^{(q)}$". By \eqref{gamma}, for each $q \geq 1$, $\text{card }\Gamma_1^{(q)}=q$ and hence by a simple argument, $\text{card }\Gamma_3^{(q)} \leq 2q \cdot m_q$.
Using the definition of $U_3^{(q)}  $ given in  \eqref{l3.4}, by Theorem  \ref{br1}  and equations \eqref{l3.1},  \eqref{c2}, and \eqref{gamma} (and using an obvious constant $C$),
\begin{align*} \bsp
&  E \left(  U_3^{(q)}   \right)^2 = E \left(  \sum_{j \in \Gamma_3^{(q)} } Y_j^{(n_q)}    \right)^2 \leq \left( \frac{1+\rho'(1)}{1-\rho'(1)}  \right)^d \sum_{j \in \Gamma_3^{(q)} } \left( s_j^{(n_q)}\right)^2\\
& \leq C \cdot \sum_{j \in \Gamma_3^{(q)} } \frac{1}{q^2m_q} \leq \frac{C \cdot 2q \cdot m_q}{q^2m_q}   \rightarrow 0 \text{ as } q \rightarrow \infty, \ q \in T. \\
\esp
\end{align*}
Therefore,
\beq\label{negu3}
 U_3^{(q)}    \rightarrow 0 \text{ in probability  as } q \rightarrow \infty, \ q \in T.
\eeq

{\bf{Step 11: }} ``A Special Blocking Argument". 
We  now return to the index sets $\Gamma_1^{(q)}$ and $\Gamma_4^{(q)}$ and the random variables $U_1^{(q)}$ and  $U_4^{(q)}$, from 
\eqref{gamma}, \eqref{l3.4}, and Steps 6, 7, and 8. We will set up (possibly ``porous") ``blocks" that alternate between indices in $\Gamma_1^{(q)}$ and $\Gamma_4^{(q)}$. We carry out this process for the case where, for a given $q\geq1$, the minimum and maximum elements of $\Gamma_1^{(q)}\cup \Gamma_4^{(q)}$ both belong to $\Gamma_4^{(q)}$. Then we will indicate the trivial changes needed for the other cases.

Suppose $q \geq 1$. Suppose that $\min \left( \Gamma_1^{(q)}\cup \Gamma_4^{(q)} \right)$ and  $\max \left( \Gamma_1^{(q)}\cup \Gamma_4^{(q)} \right)$ each belong to $\Gamma_4^{(q)}$. Recall from \eqref{gamma} that $\text{card } \Gamma_1^{(q)}=q$. For some positive integer $h(q)$ such that $h(q) \leq q$, there exists an ``alternating sequence" of nonempty, finite, (pairwise) disjoint subsets of $\Z$, namely
 $ \beta_1^{(q)}$, $\gamma_1^{(q)}$, 
$ \beta_2^{(q)}$, $ \gamma_2^{(q)}$, $\dots,  \beta_{h(q)}^{(q)}$,  $\gamma_{h(q)}^{(q)}$, and, $\beta_{h(q)+1}^{(q)}$ with the following properties:
\[  \Gamma_1^{(q)}=\bigcup_{i=1}^{h(q)} \gamma_i^{(q)};\]
\[  \Gamma_4^{(q)}=\bigcup_{i=1}^{h(q)+1} \beta_i^{(q)};\]
\[ \forall i \in \{1, 2, \ldots, h(q)\}, \ m_q+\max \beta_i^{(q)} \leq \min \gamma_i^{(q)};\]
\[ \forall i \in \{1, 2, \ldots, h(q)\}, \ m_q+\max \gamma_i^{(q)} \leq \min \beta_{i+1}^{(q)}.\]
(The last two properties come from the definition of $ \Gamma_4^{(q)}$ in \eqref{gamma}.) Next, define the following random variables:

 \begin{equation}\label{v's}
\forall i \in \{1, 2, \ldots, h(q)+1\}, \ V_i^{(q)}:=\sum_{j \in \beta_i^{(q)} }Y_j^{(n_q)} \text{ and} 
%
%
%
\end{equation}

 \begin{equation}\label{z's}
\forall i \in \{1, 2, \ldots, h(q)\}, \ Z_i^{(q)}:=\sum_{j \in \gamma_i^{(q)} }Y_j^{(n_q)}. 
%
%
%
\end{equation}
Then by \eqref{l3.4}, 
we have the following identities:
\beq\label{sumz}
U_1^{(q)}=\sum_{i=1}^{h(q)} Z_i^{(q)};
\eeq
\beq\label{sumv}
U_4^{(q)}=\sum_{i=1}^{h(q)+1} V_i^{(q)}.
\eeq

%

\noindent For a given $q \geq 1$, those notations were defined in the case where  $\min \left( \Gamma_1^{(q)}\cup \Gamma_4^{(q)} \right)$ and  $\max \left( \Gamma_1^{(q)}\cup \Gamma_4^{(q)} \right)$ both belong to $\Gamma_4^{(q)}$. In the other cases, the notations are the same, but with one or both of the following trivial changes: (i) If  $\min \left( \Gamma_1^{(q)}\cup \Gamma_4^{(q)} \right)$  belongs to $\Gamma_1^{(q)}$, then the set $\beta_1^{(q)}$ is empty and the random variable $V_1^{(q)}$ is identically 0. (ii) If $\max \left( \Gamma_1^{(q)}\cup \Gamma_4^{(q)} \right)$  belongs to $\Gamma_1^{(q)}$, then the set $\beta_{h(q)+1}^{(q)}$ is empty and the random variable $V_{h(q)+1}^{(q)}$ is identically 0.

The rest of the argument here in Step 11 will be carried out in the case where for each $q\geq 1$,  $\min \left( \Gamma_1^{(q)}\cup \Gamma_4^{(q)} \right)$ and  $\max \left( \Gamma_1^{(q)}\cup \Gamma_4^{(q)} \right)$ both belong to $\Gamma_4^{(q)}$. The changes needed in the argument to accommodate all other cases are trivial and need not be spelled out here.


For each $q \geq 1$, construct independent copies of the random variables defined in \eqref{v's} and \eqref{z's}, denoted $ \widetilde{V}_1^{(q)}$, $ \widetilde{Z}_1^{(q)}$, 
$ \widetilde{V}_2^{(q)}$, $ \widetilde{Z}_2^{(q)}$, $\dots,  \widetilde{V}_{h(q)}^{(q)}$,  $\widetilde{Z}_{h(q)}^{(q)}$, and $\widetilde{V}_{h(q)+1}^{(q)}$.  By \eqref{sumz} and Step 7, we obtain that  
\begin{align*} \bsp
\sum_{i=1}^{h(q)}Z_i^{(q)} \Rightarrow N(0, \eta_1^2)  \text{ as } q \rightarrow \infty, \ q \in T.
\esp
\end{align*}
By \eqref{c1}, the following holds:
\[\sum ^{h(q)-1}_{k=1} \alpha \left (\sigma
 \left (Z^{(q)}_i, 1\leq i \leq k  \right ), \sigma \left (Z^{(q)}_{k+1} \right )
\right ) \leq \sum ^{h(q)-1}_{k=1} \alpha ( 2m_q  ) \leq \frac{q}{q^2} \rightarrow 0 \text{ as } q \rightarrow\infty, \ q\in T.\]
Hence, by \cite{Bradley3} (Theorem 25.56),
\beq\label{l3.12}
\sum_{i=1}^{h(q)} \widetilde{Z}_i^{(q)}  \Rightarrow N(0, \eta_1^2) \text{ as } q \rightarrow \infty, \ q \in T.
\eeq
Similarly,  we obtain that 
\[\sum ^{h(q)}_{k=1} \alpha \left (\sigma
 \left (V^{(q)}_i, 1\leq i \leq k  \right ), \sigma \left (V^{(q)}_{k+1} \right )
\right ) \leq \sum ^{h(q)-1}_{k=1} \alpha ( 2m_q  ) \leq \frac{q}{q^2} \rightarrow 0 \text{ as } q \rightarrow\infty, \ q\in T,\]
and hence, 
\beq\label{l3.13}
\sum_{i=1}^{h(q)+1} \widetilde{V}_i^{(q)}  \Rightarrow N(0, \eta_4^2) \text{ as } q \rightarrow \infty, \ q \in T.
\eeq
By equations \eqref{l3.12}, \eqref{l3.13}, and independence of the random variables $ \widetilde{V}_{i}^{(q)}$,  $\widetilde{Z}_{j}^{(q)}$, with $i \in \{1, 2, \ldots, h(q)+1\} $ and $j \in \{1, 2, \ldots, h(q)\} $, we obtain that
\beq\label{l3.14}
\sum_{i=1}^{h(q)} \widetilde{Z}_i^{(q)} +  \sum_{i=1}^{h(q)+1} \widetilde{V}_i^{(q)}  \Rightarrow N(0, \eta_1^2+\eta_4^2) \text{ as } q \rightarrow \infty, \ q \in T.
\eeq
Next, for the entire ``alternating sequence" $V_1^{(q)}, Z_1^{(q)}, V_2^{(q)}, Z_2^{(q)}, \ldots, V_{h(q)+1}^{(q)}$, we note from \eqref{c1} that 
\[ 2q\cdot  \alpha ( m_q  ) \leq \frac{2q}{q^2} \rightarrow 0 \text{ as } q\rightarrow\infty, \ q\in T,\]
and applying again \cite{Bradley3} (Theorem 25.56)  and \eqref{l3.14}, we obtain the analog of \eqref{l3.14} with $\widetilde{Z}_i^{(q)}$ and $\widetilde{V}_i^{(q)}$ replaced by $Z_i^{(q)}$ and $ V_i^{(q)}$, that is, 
\beq\label{l3.15}
U_1^{(q)} + U_4^{(q)} \Rightarrow N(0, \eta_1^2+\eta_4^2)) \text{ as } q \rightarrow \infty, \ q \in T.
\eeq
Applying Slutski's theorem, by \eqref{negu2}, \eqref{negu3}, and \eqref{l3.15}, we obtain that
\begin{align}\label{last} \bsp &  \sum_{k \in B(L^{(n_q)} ) }X_k^{(n_q)} = \sum_{ k \in \text{slice}_j^{(n_q)}} Y_j^{(n_q)}= \sum_{i=1}^4 U_i^{(q)}   \Rightarrow N(0, \eta_1^2+\eta_4^2)) \text{ as } q \rightarrow \infty, \ q \in T.
\esp
\end{align}

{\bf{Step 12: }} "Convergence of Variance".
Refer to \eqref{inf}, the last paragraph of Step 6 and the last line of Step 11. To complete the proof of Lemma \ref{l3}, we now only need to show that  
\beq\label{l3.16}
\sigma^2_{n_q} \rightarrow \eta_1^2+\eta_4^2\text{ as } q \rightarrow \infty, \ q \in T.
\eeq
To accomplish that, it will (by a well know theorem) suffice to show that there is an upper bound on the fourth moments of the random variables 
$ \sum_{i=1}^4 U_i^{(q)} $, $q \in T$.

Referring to the first equality in \eqref{last}, one of course has by \eqref{normalize}, \eqref{e1.1}, and Theorem \ref{br1} that the set of numbers $\sigma_{n_q}^2, \ q \in T$ is bounded.

Since $\rho'(1)<1$, by Theorem \ref{br2}, we obtain 
(for the constant $C$ in Theorem \ref{br2}) that
\begin{align}\label{l3.17} \bsp
& E \left(  \sum_{i=1}^{4 }  U_i^{(q)} \right)^4 =E \left( \sum_{k \in B(L^{(n_q)})}  X_k^{(n)} \right)^4 \\
&\leq C \left[   \sum_{k \in B(L^{(n_q)})}   E \left( X_k^{(n)} \right)^4   +   \left( \sum_{k \in B(L^{(n_q)})}  E \left( X_k^{(n_q)} \right)^2 \right)^2  \right].
\esp
\end{align}
Using \eqref{strong} and  Theorem \ref{br1},  the first term in the right-hand side of \eqref{l3.17} can be bounded above in the following way:
\begin{align*} \bsp&   \sum_{k \in B(L^{(n_q)}) }  E \left( X_k^{(n_q)} \right)^4  =
 \sum_{k \in B(L^{(n_q)})   }  E \left[ \left( X_k^{(n_q)}  \right)^2 \left( X_k^{(n_q)}  \right)^2 \right]\\
 & \leq \theta_{n_q}^2 \sum_{k \in B(L^{(n_q)})  }  E \left( X_k^{(n_q)} \right)^2 \ll \theta_{n_q}^2 \cdot  \sigma_{n_q}^2 \rightarrow 0 \text{ as } q \rightarrow \infty, \ q \in T.
 \esp.
\end{align*}

The second term in the right-hand side of \eqref{l3.17} can be bounded above as follows: As $  q \rightarrow \infty, \ q \in T$, by Theorem \ref{br1} again,
\begin{align*}\bsp &
 \left(   \sum_{k \in B(L^{(n_q)})} E \left( X_k^{(n_q)} \right)^2 \right)^2  =   \left(   \sum_{k \in B(L^{(n_q)}) } E \left( X_k^{(n_q)} \right)^2 \right)  \left(   \sum_{k \in B(L^{(n_q)}) } E \left( X_k^{(n_q)} \right)^2 \right) \\
& \hspace{4cm} \ll \left(  \sigma_{n_q}^2 \right)^2\ll1.
\esp
\end{align*}
Hence, $\sup_{q \in T} E \left(  \sum_{i=1}^{4 }  U_i^{(q)} \right)^4 <\infty$. That completes the proof of Lemma \ref{l3}.

\epf

\section{Lindeberg Condition and Truncation}

Recall the Lindeberg condition in \eqref{e2}.  Without loss of generality, we can assume $\sigma^2_n=1$ for each $n \in \N$.  Then  by a simple argument, 
\beq\label{LT1}
\exists \epsilon_1 \geq \epsilon_2 \geq \ldots \downarrow 0 \text{ such that }  \lim_{n\rightarrow\infty} \sum_{k\in B(L_n)} E \left( X_{k}^{(n)} \right)^2I \left( \left | X_k^{(n)}\right|>\epsilon_n\right)=0.
\eeq
We truncate now at the level $\epsilon_n$. Define the following random variables:  for every $  n \in \N$ and every $k \in B(L_n)$,
\beq\label{LT2}  X^{'(n)}_k
:=X_k^{(n)}I(|X_k^{(n)}|\leq \epsilon_n)-EX_k^{(n)}I(|X_k^{(n)}|\leq \epsilon_n) \text{ and} \eeq 

\beq\label{LT3}  X^{''(n)}_k
:=X_k^{(n)}I(|X_k^{(n)}|> \epsilon_n)-EX_k^{(n)}I(|X_k^{(n)}|> \epsilon_n).\eeq 
Obviously (since $EX_k^{(n)}=0$ for each $n$ and $k$),

\beq\label{LT4}  \sum_{k \in B(L_n)} X^{(n)}_k= \sum_{k \in B(L_n)} X^{'(n)}_k+ \sum_{k \in B(L_n)} X^{''(n)}_k.
 \eeq
 Since $\rho'(1)<1$, we can apply again Theorem \ref{br1} and by \eqref{LT1}, we obtain that
 
 \begin{align*} \bsp
& 0\leq E \left(  \sum_{k \in B(L_n)} X^{''(n)}_k   \right)^2  \leq \left( \frac{1+\rho'(1)}{1-\rho'(1)}  \right)^d \sum_{k \in B(L_n)}E \left( X^{''(n)}_k \right)^2\\
& \leq C \sum_{k \in B(L_n)} E \left(X_k^{(n)} \right)^2I(|X_k^{(n)}|> \epsilon_n)   \rightarrow 0 \text{ as }n \rightarrow \infty. \\
\esp
\end{align*}
 
 Therefore, \begin{equation*}\label{LT5}
\sum_{k \in B(L_n)} X^{''(n)}_k   \rightarrow 0 \text{ in probability as }n \rightarrow \infty. \
  \end{equation*}
 As a consequence, by Slutski's theorem, to prove that 
 \beq\label{LTG}
\sum_{k \in B(L_n)} X_k^{(n)} \Rightarrow
N(0, 1) \text{ as } n\rightarrow\infty,
\eeq
we only have left to show that
 
 \begin{equation}\label{LT6}
\sum_{k \in B(L_n)} X^{'(n)}_k   \Rightarrow N(0, 1) \text{ as } n \rightarrow \infty. \
  \end{equation}
 
 Note that $\| X^{'(n)}_k   \|_{\infty} \leq 2\epsilon_n$ for every $n \in \N$ and every $k \in B(L_n)$.  Since $\epsilon_n \rightarrow 0$ as $n \rightarrow  \infty$ by \eqref{LT1}, we have that 
 \[ \sup_{k \in B(L_n)}\| X^{'(n)}_k   \|_{\infty}  \rightarrow 0 \text{ as } n \rightarrow  \infty. \]
Hence by Lemma \ref{l3}, \eqref{LT6} holds, and hence also \eqref{LTG}.  The proof of Theorem \ref{r} is complete.

\section{Generalization}
\begin{Thm}\label{gr} Suppose $d$ is a positive integer. For each $n \in N$, suppose $L_n:=(L_{n1}, L_{n2}, \ldots, L_{nd})$ is an element of $N^d$, and suppose $X^{(n)}:=\left( X_k^{(n)}, k \in B(L_n) \right)$  is an array of random variables such that for each $k \in B(L_n)$,  $EX_{k}^{(n)}=0$ and  $E \left( X_{k}^{(n)} \right)^2<\infty$, and for at least one $k \in B(L_n)$, $E \left( X_{k}^{(n)} \right)^2>0$.  Suppose also that the mixing assumptions \eqref{e1} and
\begin{equation}\label{e1.1.0}\lim_{m\rightarrow \infty}\rho'(m)<1 \eeq
hold, where for each $m \in \N$, 
\[ \rho'(m):=\sup_{n \in N} \rho'(X^{(n)}, m).\]
For each $n \in \N$, define the random sum $S \left( X^{(n)}, L_n \right):=\sum_{k \in B(L_n)} X_k^{(n)}$ and define the quantity $\sigma_n^2:= E \left( S \left( X^{(n)}, L_n \right) \right)^2$. 
Suppose there exists a positive constant $C$ such that for every $n \in \N$ and every nonempty set $ S \in B(L_n)$, 
\begin{equation}\label{e2-1}
E \left( \sum_{ k \in S} X_k^{(n)} \right)^2 \geq C \cdot \sum_{k \in S} E \left( X_k^{(n)} \right)^2.
\end{equation}
Suppose the Lindeberg condition \eqref{e2} holds. 
Then \[\sigma_n^{-1}S(X^{(n)}, L_n)\Rightarrow
N(0, 1) \text{ as } n\rightarrow\infty.\] 
\end{Thm}
For $d=1$, this result was proved by Peligrad (\cite{MP}, Theorem 2.1), with \eqref{e2-1} replaced by a weaker assumption. The proof of Theorem \ref{gr} again involves induction  on the dimension $d$, and is just a slight modification of the argument in  Sections 3, 4, and 5 for Theorem \ref{r}. In essence, in place of \eqref{e1.1} and Theorem \ref{br1}, one uses \eqref{e1.1.0}, Theorem \ref{br0}, and \eqref{e2-1}.

In fact, to make that argument work smoothly, it suffices to have a weaker version of \eqref{e2-1} in which, for a given $n \in \N$, the sets $S \subseteq B(L_n)$ are restricted to certain special ``rectangles" of the form $S=S_1\times S_2 \times \ldots \times S_d$ where for each $ j \in \{ 1, 2, \ldots, d\}$, the set $S_j$ either is $\{1, 2, \ldots, L_{n_j} \}$ or is $\{k\}$ for some $k \in \{1, 2, \ldots, L_{n_j} \}$.



\bibliographystyle{spmpsci}      
\bibliography{References}   

\begin{thebibliography}{1}
\providecommand{\url}[1]{{#1}}
\providecommand{\urlprefix}{URL }
\expandafter\ifx\csname urlstyle\endcsname\relax
  \providecommand{\doi}[1]{DOI~\discretionary{}{}{}#1}\else
  \providecommand{\doi}{DOI~\discretionary{}{}{}\begingroup
  \urlstyle{rm}\Url}\fi

\bibitem{Bradley3}
Bradley, R.C.: Introduction to {S}trong {M}ixing {C}onditions, vol. 1, 2, \&3.
\newblock Kendrick Press, Heber City (Utah) (2007)

\bibitem{Mil}
Miller, C.: Three theorems on $\rho^*$-mixing random fields.
\newblock J. Theor. Probab. \textbf{7}, 867--882 (1994)

\bibitem{MP}
Peligrad, M.: On the asymptotic normality of sequences of weak dependent random
  variables.
\newblock J. Theor. Probab. \textbf{9}, 703--715 (1996)

\bibitem{Pelig2}
Peligrad, M.: Maximum of partial sums and an invariance principle for a class
  of weak dependent random variables.
\newblock Proc. AMS \textbf{126}, 1181–--1189 (1998)

\bibitem{UP}
Peligrad, M., Utev, S.A.: Maximal inequalities and an invariant principle for a
  class of weakly dependent random variables.
\newblock J. of Theor. Probab. \textbf{16}(1), 101--115 (2003)

\bibitem{Pro}
Prohorov, Y.V.: Convergence of random processes and limit theorems in
  probability theory.
\newblock Theor. Probability Appl. \textbf{1}, 157--214 (1956)

\bibitem{Tone4}
Tone, C.: Kernel density estimators for random fields satisfying an interlaced
  mixing condition.
\newblock Journal of Statistical Planning and Inference \textbf{143},
  1285--1294 (2013)

\end{thebibliography}

%
%

\end{document}